\documentclass{article}
\usepackage{color}

 \newcommand {\theoremstyle} [1] { }
 \newenvironment{proof}{{\noindent\it\underline{Proof}}:}{\hfill$\Box$}
  
 \newtheorem{thm}{Theorem}[section]
  
 \newtheorem{prop}{Proposition}[section]
 \theoremstyle{plain}
 \newtheorem{lem}[thm]{Lemma} 
 \newtheorem{cor}{Corollary}[section]
 \theoremstyle{definition}

 \theoremstyle{remark}
 \newtheorem{rem}[thm]{Remark}

\usepackage {amssymb}
\usepackage{authblk}

\def\R{\mathbb{R}}
\def\N{\mathbb{N}}
\def \ee{\varepsilon}

\usepackage {amssymb}

\makeatother

\begin{document}

\author[1]{P. Amster}
\author[1]{M. P. Kuna}
\author[2]{G. Robledo}
\affil[1]{Departamento de Matem\'atica, Facultad de Ciencias Exactas y Naturales, Universidad de Buenos Aires, Argentina and IMAS-CONICET.}
\affil[2]{Departamento de Matem\'aticas, Facultad de Ciencias,  Universidad de Chile, Chile.}

\title{Multiple solutions for periodic 
perturbations of a delayed autonomous system near an
equilibrium }
\date{}
\maketitle

\begin{abstract}
    
Small non-autonomous perturbations 
around an equilibrium of a nonlinear delayed system are 
studied. 
Under appropriate assumptions, it is shown that the number of 
$T$-periodic solutions lying inside a bounded domain $\Omega\subset \R^N$ is, generically,
at least $|\chi \pm 1|+1$, 
where $\chi$ denotes the Euler characteristic of $\Omega$. 
Moreover, some connections between the associated fixed point operator and the Poincar\'e operator are explored.

\end{abstract}

\section{Introduction}
Let 
$\Omega\subset \R^N$ be a bounded domain with 
smooth boundary. An elementary  result from the theory of ODEs 
establishes that if a smooth function $G:\overline\Omega\to \R^N$
is inwardly pointing over $\partial\Omega$, that is 
\begin{equation}
\label{hart-weak}
\langle G(x),\nu(x)\rangle <0 \qquad x\in \partial\Omega,
\end{equation}
where $\nu(x)$ denotes the outer
normal at $x$, then the solutions of the autonomous system of ordinary differential equations 
$$u'(t)=G(u(t))$$ with initial data $u(0)=u_0\in \overline \Omega$ are defined and remain inside $\Omega$ for all $t>0$. 

{Now, let us denote the space of $T$--periodic continuous functions as
$$
C_T:=\{u\in C(\R,\R^N):u(t+T)=u(t)\}
$$
and, for given $p\in C_{T}$, consider
the non-autonomous system
$$u'(t)=G(u(t)) + p(t).$$}
 
{If 
$\overline\Omega$ has the fixed point property, then the above system has at least one $T$-periodic orbit,  
provided that  $\|p\|_\infty$ is small.} This is a straightforward  consequence of the fact 
that the time-dependent vector field $G(x)+ p(t)$ is still inwardly pointing for all $t$; 
hence, the set $\overline \Omega$ is invariant for the associated flow and thus
the Poincar\'e operator given by $Pu_0:=u(T)$ is 
well defined for 
 $u_0\in \overline\Omega$ 
and satisfies $P(\overline\Omega)\subset \overline\Omega$.

More generally, observe that, when (\ref{hart-weak}) is assumed, the homotopy defined by $h(x,s):= sG(x) - (1-s)\nu(x)$ with $s\in [0,1]$ 
does not vanish on $\partial\Omega$; whence
$$
deg_B(G,\Omega,0) = deg_B(-\nu,\Omega,0), 
$$
where $deg_B$ stands for the Brouwer degree. Thus, it follows from \cite{hopf} 
that 
$deg_B(G,\Omega,0)=(-1)^N\chi(\Omega)$, where 
$\chi(\Omega)$ denotes the Euler characteristic of $\Omega$.

It is worthy to recall 
(see \emph{e.g.},\cite{wecken}) that if $\overline \Omega$ has the fixed point property, 
then $\chi(\Omega)$ is different from $0$. This follows easily in the present 
setting from the fact that if $\chi(\Omega)=0$ then one can 
construct  a field 
$G$ satisfying (\ref{hart-weak})  that does not vanish in $\Omega$. If $\overline\Omega$ has the fixed point property, then there exist (non-constant) $T$-periodic solutions of all periods which, in turn, implies that $G$ vanishes, a contradiction. 
Interestingly, the converse 
of the result in \cite{wecken} is not true; that is, one can easily find $\Omega$ with nonzero Euler characteristic such that $\overline \Omega$ has not the fixed point property. For such a domain, the Poincar\'e map has obviously a fixed point (because $G$ vanishes in $\Omega$). 
This yields the conclusion that a fixed point-free map in $C(\overline \Omega,\overline\Omega)$ cannot belong to the closure of the 
set of all the  
Poincar\'e maps associated to the homotopy class of $-\nu$.

Now suppose, independently of the value of $\chi(\Omega)$, that 
$G$ vanishes at some point 
$e\in \Omega$, namely, that 
$e$ is an equilibrium point of the autonomous system. 
It is well known that if $M:=DG(e)$ is nonsingular, 
then the degree of $G$ over any small neighbourhood $V$ of $e$ is well defined and coincides with $s(M)$, where 
\begin{equation}
\label{sM} 
s(M):= sgn ({\rm det}(M)).
\end{equation}
Thus, if $s(M)$ is different from 
$(-1)^N\chi(\Omega)$, then 
the excision property of the degree implies that the system has at least another equilibrium point in 
$\Omega\setminus \overline V$. Furthermore, 
it follows from Sard's lemma 
that, for almost all values $\overline p$ in a neighbourhood of $0\in \R^N$, 
the mapping  $G + \overline p$ has at least $\Gamma$ different
zeros in $\Omega$, with
\begin{equation}
\label{Gamma} 
\Gamma=\Gamma(M):=|\chi(\Omega)- (-1)^{N} s(M)| + 1.
\end{equation}

Thus, one might expect that if $p\in C(\R,\R^N)$ is $T$-periodic and $\|p\|_\infty$ is small, then the number of $T$-periodic solutions of the non-autonomous system 
is generically greater or equal to $\Gamma$. Here, `generically' should be understood in the sense of Baire category, that is, the property is valid for all $p$ (close to the origin) in the space of continuous $T$-periodic except for a meager set. 
It can be shown, indeed, that  the fixed point index of the Poincar\'e map $P$ at $e$ is equal to $(-1)^Ns(M)$ and, moreover, a homotopy argument shows that the 
degree of $P$ over $\Omega$ is equal to $\chi(\Omega)$. Details are omitted 
because the result follows from the main theorem of the present paper.

For several reasons, the situation is different for the delayed system 
\begin{equation}
\label{ec}
u'(t) = g(u(t),u(t-\tau)) 
\end{equation}
where, for simplicity, we shall assume that $g:\overline\Omega\times \overline\Omega\to \R^N$
is continuously differentiable.
In the first place observe that, due to the delay, the condition 
that the field $G(x):=g(x,x)$ is inwardly pointing does not 
necessarily avoid that solutions 
with initial data $x_0:=\phi\in C([-\tau,0],\overline\Omega)$
may eventually abandon $\overline\Omega$. 
However,
taking into account that
$$
|u(t_0-\tau)- u(t_0)|
\le \tau \max_{t\in [t_0-\tau,t_0]} |u'(t)|,$$
it follows that the flow-invariance property, 
now over the set $C([-\tau,0],\overline\Omega)$, is retrieved under the stronger assumption

\begin{equation}
\label{hart}
\langle g(x,y),\nu(x)\rangle < 0 \qquad (x,y)\in 
\mathcal A_\tau 
(\Omega)
\end{equation}
where 
$$
\mathcal A_\tau 
(\Omega):= \{ (x,y)\in \partial\Omega\times \overline\Omega: |y-x|\le \tau\|g\|_{\infty}\}. 
$$

In the second place, the previous considerations regarding the 
Poincar\'e map become less obvious, 
since the latter is now defined not over $\overline\Omega$ but 
over the metric space
$C([-\tau,0],\overline\Omega)$.
In connection with this fact, we recall that the characteristic equation for the autonomous linear delayed systems is 
transcendental (also called quasipolynomial 
equation), so there exist typically   infinitely many complex characteristic values.

\medskip

Throughout the paper, we shall assume as before that system (\ref{ec}) has an equilibrium point 
$e\in \Omega$, that is, such that  $g(e,e)=0$. 
This necessarily occurs  
when $\chi(\Omega)\neq 0$, although this latter condition shall not be imposed. 

Denote by $A,B\in \R^{N\times N}$
the 
respective matrices $D_xg(e,e)$ and $D_yg(e,e)$. Again, if $A+B$ is nonsingular and
$s(A+B)$ is different from 
$(-1)^N\chi(\Omega)$, then
the system has at least one extra  
equilibrium point in $\Omega$; 
furthermore, 
the number of equilibria in $\Omega$ is
generically greater or equal to $\Gamma$. 
This is readily verified by writing the set of all the functions 
$g\in C^1(\overline\Omega\times\overline\Omega,\R^N)$ satisfying (\ref{hart}) 
as the union of the closed sets 
$$X_n:=\left\{g\in C^1(\overline\Omega\times\overline\Omega,\R^N): 
\langle g(x,y),\nu(x)\rangle \le -\frac 1n \quad\hbox{for $(x, y)\in \mathcal A_\tau 
(\Omega)$}
\right\}$$
and noticing that 
$X_n\cap \mathcal C$
is nowhere dense, where $ \mathcal C$ denotes the set of those functions $g$ such that $0$ is a critical value of the corresponding $G$.

Our goal in this work is to extend  the preceding ideas for 
non-autonomous periodic perturbations of (\ref{ec}), namely the problem 
\begin{equation}
    \label{nonaut}
u'(t) = g(u(t),u(t-\tau)) + p(t)    
\end{equation}
with {$p\in C_{T}$}. 

As a basic hypothesis, we shall assume that 
the linearisation at the equilibrium, that is,  the system  
\begin{equation}
\label{linear}
u'(t) = Au(t)+ Bu(t-\tau)
\end{equation}
has no nontrivial $T$-periodic solutions. 
This clearly implies, in particular, the above-mentioned 
condition that $A+B$ is invertible. 
From the Floquet theory for DDEs, 
it is known that the latter condition is also 
sufficient for nearly all positive values of $T$ (\emph{i.e.}, except for at most a countable set). For the sake of completeness, this specific consequence of the Floquet theory shall be 
shown below (see Remark \ref{remark1}).

Our main result reads as follows.

\begin{thm}
\label{main}
Let the equilibrium $e$ and the matrices $A$ and $B$ be as before and 
assume that 
the linear system (\ref{linear})
has no nontrivial $T$-periodic solutions. Then:
\begin{itemize}
\item[(a)] There exists $r>0$ such that 
{for any $p\in 
C_{T}$}
with $\|p\|_\infty<r$  the non-autonomous problem
(\ref{nonaut})
has at least one $T$-periodic solution. 
\item[(b)] If moreover 
(\ref{hart}) holds and 
$
s(A+B) \neq (-1)^N\chi(\Omega)
$
with $s$ defined as in (\ref{sM}), then
(\ref{nonaut})
has at least two $T$-periodic solutions. 
\item[(c)] Furthermore, there exists
a residual set $\Sigma_r\subset C_T$ such that if
$p\in \Sigma_r\cap B_r(0)$, then the number of
$T$-periodic solutions
is at least $\Gamma(A+B)$, where $\Gamma$ is given by (\ref{Gamma}).
\end{itemize}
\end{thm}

The next result is an immediate consequence of {Theorem \ref{main} combined with the preceding comments}.

\begin{cor}
\label{corol}

Let $e, A$ and $B$ be as before and 
assume that 
$A+B$ is invertible. 
Then for nearly all $T>0$ there 
exists $r=r(T)>0$ such that if {$p\in 
C_{T}$} with $\|p\|_\infty<r$ then the non-autonomous problem
(\ref{nonaut})
has at least one $T$-periodic solution. 
If moreover (\ref{hart}) holds and 
$
s(A+B) \neq (-1)^N\chi(\Omega), 
$
then the number of $T$-periodic solutions is 
at least $2$ and generically $\Gamma(A+B)$.

\end{cor}

{For small delays,
the condition that (\ref{linear})} has no nontrivial $T$-periodic solutions 
can be formulated explicitly in terms of the matrix $A+B$ :

\begin{cor}
\label{smalldelay}

Let $e, A$ and $B$ be as before 
and assume that
$\frac{2k\pi}Ti$ is not an eigenvalue of 
the matrix $A+B$  for all $k\in\N_0$. 
Then for each $\tau$ small enough there 
exists $r=r(\tau)$ 
such that 
the non-autonomous problem
(\ref{nonaut})
has at least one $T$-periodic solution for any {$p\in 
C_{T}$} with $\|p\|_\infty<r$. 
If moreover  (\ref{hart-weak}) holds for $G(x):=g(x,x)$ and $s(A+B)\ne (-1)^N\chi(\Omega)$, then 
(\ref{nonaut})
has at least
two 
$T$-periodic solutions and generically $\Gamma(A+B)$. 
\end{cor}

It is worthy mentioning that if $\Omega$ is for example a ball, then the condition 
$s(A+B)\neq (-1)^N\chi(\Omega)$
implies that the equilibrium is unstable. 
As we shall see, this can be regarded 
as a consequence of the fact that the Leray-Schauder index of 
the fixed point operator defined in the proof of our main theorem is $(-1)^{N+1}$. 
This connection 
can be deduced from  a version of
the Krasnoselskii relatedness principle, which implies that the mentioned index coincides except for a $(-1)^N$ factor with that of the Poincar\'e operator. As shown in
 Proposition \ref{poinc-stab}, this implies, in turn, that the equilibrium cannot be stable.

The paper is organised as follows. 
In the next section, we prove some basic facts concerning the linearised problem (\ref{linear}); in particular, we give a necessary and sufficient condition in order
to ensure that it has no nontrivial $T$-periodic solutions. 
In section \ref{dem} we present a proof of Theorem \ref{main} by means of an appropriate fixed point operator. The next two sections are devoted to a proof In section \ref{sec-delay}, 
we give a proof Corollary \ref{smalldelay}. In section \ref{poincare}, we make some 
considerations on the stability of the equilibrium and the indices, on the one hand, of the fixed point operator defined in section \ref{dem} and of the Poincar\'e map, on the other hand.
Finally, a simple application of the main results for a singular system is introduced in section \ref{exam}.

\section{Linearised system}

In this section, we shall prove some basic facts concerning the linear system (\ref{linear}). 
To this end, let us introduce some notation. 
For $k\in \mathbb N_0$, define
$$\lambda_k:= \frac{2k\pi}T$$
and
$$\varphi_k(t):= \cos(\lambda_k t) \qquad 
\psi_k(t):= \sin (\lambda_k t).$$
It is readily verified that
$$
\varphi_k(t-\tau)= \varphi_k(t)\varphi_k(\tau) +
\psi_k(t)\psi_k(\tau)
$$
$$
\psi_k(t-\tau)= \psi_k(t)\varphi_k(\tau) -
\varphi_k(t)\psi_k(\tau)
$$
and 
$$
\varphi_k'= -\lambda_k \psi_k,\qquad \psi_k'=\lambda_k\varphi_k.
$$

For an element $u\in C_T$, we may consider its Fourier series, namely
$$u = a_0 + 
\sum_{k=1}^\infty (\varphi_k a_k +\psi_k b_k)
$$
in the $L^2$ sense, with $a_k, b_k\in \R^N$.
Furthermore, recall that if $u$ is smooth (\emph{e.g.}, of class $C^2$) then the series and 
its term-by-term derivative converge uniformly to $u$ and $u'$ respectively.

\begin{lem}
\label{lema}
Let $u\in C_T$ and define
\begin{equation}
\label{matrices}
X_k:=A+\varphi_k(\tau)B, 
\qquad Y_k:=\lambda_kI + \psi_k(\tau) B.
\end{equation} 
Then $u$ is a solution of 
(\ref{linear})
if and only if
\begin{equation}
\label{matr-ident}
\left(\begin{array}{cc}
X_k & -Y_k\\
Y_k & X_k
\end{array} 
\right)
\left(\begin{array}{c}
a_k\\
b_k
\end{array} 
\right) = 
\left(\begin{array}{c}
0\\
0
\end{array} 
\right)
\end{equation}
for all $k\in \mathbb N_0$. 
\end{lem}

\begin{proof}
Since 
$\varphi_k'(t), \varphi_k(t-\tau), \psi_k'(t)$ and 
$\psi_k(t-\tau)$ belong to ${\rm\bf span}\{ \varphi_k(t),\psi_k(t)\}$, 
it follows that 
$u$ is a solution of 
of (\ref{linear}) if and only if
$$
(A+B)a_0=0
$$
and
$$
\varphi_k'(t) a_k +\psi_k'(t) b_k =
A(\varphi_k(t) a_k +\psi_k(t) b_k)
+ B(\varphi_k(t-\tau) a_k +
\psi_k(t-\tau) b_k)
$$
for all $k>0$. 
The latter identity, in turn, is equivalent to

$$
\begin{array}{ccc}
 \lambda_k b_k 
& = &  [A+ \varphi_k(\tau) B] a_k - \psi_{k}(\tau) Bb_k
\\
{}
\\
-\lambda_k a_k
& =  & \psi_{k}(\tau)Ba_k + 
[A+ \varphi_k(\tau) B] b_k,
\end{array}.
$$
that is, 
$$X_ka_{k} -Y_kb_{k}= Y_ka_{k} + X_kb_{k}=
0.
$$
Because $X_0=A+B$ and $Y_0=0$, we deduce that 
$u$ is a solution  
of (\ref{linear}) if and only if (\ref{matr-ident}) holds for all $k\in \mathbb N_0$.

\end{proof}

\begin{cor}
\label{no-nontrivial}

 (\ref{linear}) has no nontrivial $T$-periodic solutions 
if and only if
\begin{equation}
\label{nec-suf}    
h_k:={\rm det}\left(\begin{array}{cc}
X_k & -Y_k\\
Y_k & X_k
\end{array} 
\right)\neq 0
\end{equation}
for all $k\in \mathbb N_0$. 

\end{cor}

\begin{rem}
\label{remark1}

\

\begin{enumerate}
    \item 

Because $A+B$ is invertible, it is clear that
for nearly all $T>0$ condition 
(\ref{nec-suf})  
is satisfied for all $k$. Indeed, it 
suffices to observe that 
$h_k$, regarded as a function of $T\in (0,+\infty)$, is an analytic function and, consequently, it has at most a countable number of zeros.  

\item It can be shown that 
$h_k\ge 0$; in particular, 
its roots have even multiplicity.

The proof is straightforward  
when $A$ and $B$ commute,  
since in this case
$$
{\rm det}\left(\begin{array}{cc}
X_k & -Y_k\\
Y_k & X_k
\end{array} 
\right)
= 
{\rm det}(X_k ^2+Y_k ^2).
$$
The conclusion then follows, because for any pair of square real matrices $X, Y$ such that $XY=YX$ it is verified that
$$
{\rm det}(X ^2+Y^2)= {\rm det}[(X+iY)(X-iY)] =
{\rm det}(X+iY)\overline{{\rm det}(X+iY)}\ge 0.  
$$
{A proof for the non-commutative case is given below in section \ref{dem}, step \ref{directo-fourier}.  
}

It is noticed that (\ref{nec-suf}) may hold for 
non-invertible matrices $X_k$ and $Y_k$: for instance, observe that 
$$
\left(
\begin{array}{cc}
1 & 0 \\
0 & 0
\end{array}\right)^2 +  
\left(
\begin{array}{cc}
0 & 0 \\
0 & 1
\end{array}\right)^2 = I.
$$

\item 
\label{k_0}
Since $\lambda_k\to +\infty$ it follows, for $k$ large, that 
$$h_k={\rm det}(Y_k){\rm det}(Y_k + X_kY_k^{-1}X_k)\simeq
\lambda_k^{2N} > 0.$$
In particular, 
there exists $k_0$ such that if $u$ is a $T$-periodic solution of (\ref{linear}) then $a_k=b_k=0$ for $k>k_0$. 
This means that $u$ is a (vector) trigonometric polynomial. 
Incidentally,  
observe that, because the family $\{\varphi_k, \psi_k\}$ is uniformly bounded, the constant $k_0$ may be chosen independent of $\tau$.

{
In other words, if we consider the linear operator $L:C_T\to C_T$, given by $Lu(t):=u'(t) - Au(t) - Bu(t-\tau)$, then $
{\rm ker}(L)\subset {\rm \bf span}\{\varphi_k,\psi_k\}_{0\le k\le k_0}$. 
Observe furthermore that
${\rm Im}(L)$ consists of all the Fourier series $a_0+
\sum_{k>0}(\varphi_ka_k + \psi_kb_k)$ such that $a_0\in 
{\rm Im}(A+B)$ and $(a_k,b_k)\in {\rm Im}(M_k)$, where $M_k$ is the matrix defined in (\ref{matr-ident}). This yields a direct proof of the well-known fact that $L$ is a zero-index Fredholm operator.
Moreover, it is verified that 
$(a_k,b_k)\in {\rm ker}(M_k)\iff (-b_k,a_k) \in {\rm ker}(M_k)$,
a fact that will be of  relevance in the proofs of our results.}

\end{enumerate}
\end{rem}

\section{Proof of the main theorem}
\label{dem}

For convenience, a little of extra notation shall be introduced. 
For a function $u\in C_T$, let us
write
$$\mathcal Iu(t):= 
\int_0^t u(s)\, ds, \qquad 
\overline u:= \frac 1T {\mathcal Iu (T)}.
$$

Moreover, denote by $\mathcal N$ the Nemitskii operator associated to the problem, namely 
$$
\mathcal Nu(t):= g(u(t),u(t-\tau)). 
$$

Without loss of generality we may assume 
$e=0$ and fix $T>0$ such that (\ref{linear}) has no nontrivial 
$T$-periodic solutions. For simplicity, we shall assume from the beginning that 
all the assumptions are satisfied; it shall be easy for the reader to deduce the existence of one solution near the equilibrium when (\ref{hart}) is not satisfied. 

Define the open bounded set 
$U=\{u\in C_T:u(t)\in \Omega\,\hbox{ for all $t$} \}$ and the compact operator 
$K:\overline U\to C_T$ defined by
$$
Ku(t):= \overline u -t\, \overline {\mathcal Nu} + 
\mathcal I\mathcal Nu(t) - \overline{\mathcal I\mathcal Nu}.
$$

We shall prove that the 
Leray-Schauder degree of $I-K$ is equal to $(-1)^N\chi(\Omega)$
over $U$ and to $s(A+B)$ over $B_\rho(0)$ for small values of 
$\rho>0$.

To this end, let us proceed in several steps: 

 \begin{enumerate}
\item 

Let $K_0u:= \overline u -\frac T2 
\overline {\mathcal Nu}$ and define, for $s\in [0,1]$, 
the
operator given by 
$K_s:=s K +(1-s)K_0$. We claim that $K_s$ 
has no fixed points on $\partial U$.
Indeed, 
for $s>0$ it is clear that $u\in\overline U$ 
is a fixed point of $K_s$ if and only if
$u'(t)=s\mathcal Nu(t)$, that is:
$$
u'(t)= sg(u(t),u(t-\tau)).
$$

Suppose there exists  $t_0$ such that 
$u(t_0)\in\partial\Omega$, then we deduce, as before,
$$|u(t_0-\tau)- u(t_0)|
\le \tau \max_{t\in [t_0-\tau,t_0]} |u'(t)|
\le \tau \|g\|_
\infty
$$
and by (\ref{hart})
we obtain
$$0=
\langle u'(t_0),\nu(u(t_0))\rangle =
s\langle g(u(t_0),u(t_0-\tau)),\nu(u(t_0))\rangle 
<0,
$$
a contradiction. 
On the other hand, 
we observe that the range of
$K_0$ is contained in the set of constant functions, which 
can be identified with 
 $\R^N$; thus, the Leray-Schauder degree of $I-K_0$ can be computed as the Brouwer degree of its restriction to  
$\overline U\cap \R^N = \overline \Omega$. 

Furthermore, for $u(t)\equiv u\in \overline \Omega$ it is clear that $K_0u= u - \frac T2 G(u)$, which does not vanish on $\partial \Omega=\partial U\cap \R^N$. 
By the homotopy invariance of the degree, we conclude that
$$deg(I-K,U,0)=deg \left(\frac T2G,\Omega,0\right)=(-1)^N\chi(\Omega).$$

\item 
Let $K_L$ be the operator associated to the 
linearised problem, defined by
$$
K_Lu(t):= \overline u -t\,\overline {\mathcal N_Lu} + 
 \mathcal I\mathcal N_Lu(t) - \overline{\mathcal I\mathcal N_Lu},
$$
with $\mathcal N_Lu(t):= Au(t) + Bu(t-\tau).$
As before, it is seen that $K_Lu=u$ if and only if $u$ is a solution of (\ref{linear}); hence, it follows from the assumptions that $K_L$ has no nontrivial fixed points. 

Furthermore, the degree of $I-K_L$ coincides with 
the degree of $I-K$ on $B_\rho(0)$ when
$\rho$ is small. This is a well-known fact but, for the reader's convenience, a simple proof is sketched as follows.

Since the degree is locally constant, we may assume that  
$g$ is of class $C^2$ near $(0,0)$, then {for some $C>0$,}
$$
\|Kv-K_Lv\|_{\infty} \le C\|\mathcal Nv-\mathcal N_Lv\|_\infty 
= o(\rho).
$$
Because $K_L$ is compact, it is verified that, for some 
$\theta>0$, 
$$
\|v-K_Lv\|_{\infty}\ge \theta \rho
$$
for all $v\in \partial B_\rho(0)$. 
Indeed, due to linearity, it suffices to prove the claim for $\rho=1$. 
By contradiction, suppose there exists a sequence $\{v_n\}\subset \partial B_1(0)$ such that $\|v_n-K_Lv_n\|_{\infty}\to 0$, then passing to a subsequence we may assume that $\{K_Lv_n\}$ 
converges to some $v$. 
Then $v_n\to v$ which, in turn, implies that $\|v\|_{\infty}=1$ and $v=K_Lv$, a contradiction.   
It follows that if 
$\rho>0$ is small then 
 $sK + (1-s)K_L$ has no fixed points on $\partial B_\rho(0)$ for $s\in [0,1]$ because 
$$
\|v - sKv - (1-s)K_Lv\|_{\infty} \ge \|v - K_Lv\|_{\infty} - \|K_Lv-Kv\|_{\infty}
\ge \theta \rho - o(\rho)>0$$ 
for $v\in \partial B_\rho(0)$. Thus, the degree of $I-K$ is well defined and coincides with the degree of $I-K_L$ over $B_\rho(0)$.

\item Claim: $deg(I-K_L,B_\rho(0),0) = s(A+B)$.

\label{directo-fourier}

Indeed, for $u$ as before it is seen by direct computation that 
$$u-K_Lu=\tilde a_0 + \sum_{k\ge 1} (\varphi_k\tilde a_k + \psi_k\tilde b_k)$$
where
$$\tilde a_0= \mathcal M_0 a_0
$$
and
$$
\left(
\begin{array}{c}
\tilde a_k       \\
\tilde b_k      
\end{array}\right) = \mathcal M_k
\left(
\begin{array}{c}
 a_k       \\
b_k      
\end{array}\right)
$$
with
$$\mathcal M_0:= \frac T2(A+B)\qquad \hbox{and }\,\,   
\mathcal M_k:= \frac 1{\lambda_k}
\left(
\begin{array}{cc}
Y_k  & X_k      \\
-X_k & Y_k      
\end{array}\right)\quad \hbox{for}\,  k>0.
$$
Hence, the degree coincides with the sign of the determinant of the block matrix
$$
\left(
\begin{array}{ccccc}
\mathcal M_0 & 0 & 0 & \ldots & 0       \\
0 & \mathcal M_1 & 0 & \ldots & 0 \\       
0 & 0 & \mathcal M_2 & \ldots & 0\\   
\ldots & \ldots & \ldots & \ldots & \ldots\\
0 & 0 & 0 & \ldots & \mathcal M_J       
\end{array}\right)
$$
for $J$ sufficiently large. 
Thus, the proof follows in a straightforward manner from the fact that 
${\rm det}(\mathcal M_k) >0$ for all $k>0$. 
We remark that the latter property holds even when $A$ and $B$ do not commute
(see Remark \ref{remark1}).

Indeed, identifying the pairs $(a,b)\in \R^N\times \R^N$ with vectors 
$a+ib\in \mathbb C^N$,  a matrix of the form 
$\left(
\begin{array}{cc}
X & -Y      \\
Y & X      
\end{array}\right)$ may be called a {$\mathbb C$-linear matrix}.  Thus, we need to prove that if $\mathcal M$ is an arbitrary invertible $\mathbb C$-linear matrix, then  the algebraic multiplicity of each eigenvalue $\sigma<0$ of $\mathcal M$ is even. 
It is known that 
this value can be computed as the dimension of the kernel of the matrix $(\mathcal M-\sigma I)^m$, where $m$ is the minimum integer such that ${\rm ker}(\mathcal M-\sigma I)^m = {\rm ker}(\mathcal M-\sigma I)^{m+1}$.
Now observe that the set of $\mathbb C$-linear matrices is a subring of $\R^{2N\times 2N}$; thus, $(\mathcal M-\sigma I)^m$ is again a 
$\mathbb C$-linear matrix. 
In particular, if $(a,b)\in {\rm ker}(\mathcal M-\sigma I)^m$ then $(-b,a)\in {\rm ker}(\mathcal M-\sigma I)^m$
and the result follows.

\item \textit{Existence of two solutions for small $p$}. 

From the previous steps and the fact that the degree 
is locally constant we deduce that
$$deg(I-K,U,\hat p)=(-1)^N\chi(\Omega),\qquad deg(I-K,B_\rho(0),\hat p)=s(A+B)$$
when $\|\hat p\|_\infty$ is small. 
Now the excision property of the Leray-Schauder degree implies 
$$deg(I-K,B_\rho(0),\hat p)=s(A+B)\ne 0,$$
and $$ 
deg(I-K,U\backslash B_\rho(0),\hat p)=(-1)^N\chi(\Omega)-  s(A+B) 
\ne 0.$$

Thus, there exists $\hat r>0$ such that 
the equation $(I-K)u=\hat p$ has at least two solutions for
$\|\hat p \|_\infty <\hat r$.
Finally, for each $p\in C_T$ define 
$$\hat p (t):= \mathcal I p(t) - \overline{\mathcal Ip} - t\overline p,
$$
then clearly $\|\hat p\|_\infty\le c\|p\|_\infty$ for {some $c>0$}. 
The result is then deduced from the fact that if
$u-Ku=\hat p$, then $u$ is a $T$-periodic solution of {(\ref{nonaut})}.
$$u'(t)=g(u(t),u(t-\tau))+p(t).$$

\item 
\textit{Genericity.}


The last part of the proof follows as a consequence of the following particular case of the Sard-Smale Theorem \cite{smale}:
\begin{thm}
Let $\mathcal F:X \to Y$ be a $C^1$ Fredholm map of index $0$ between Banach manifolds, i.e. such that  
$D\mathcal F(x):T_x X \to T_{\mathcal F(x)} Y$ is a Fredholm operator of index $0$ for every $x\in X$.
Then the set of regular values of $\mathcal F$ is residual in $Y$.
\end{thm}

At this point, we notice that the argument
 is a bit subtle: when applied
to $\mathcal F:=I-K$, the Sard-Smale Theorem implies the 
existence of a residual set $\Sigma\subset C_T$ such that the mapping $\mathcal F-\hat p$ has at least $\Gamma - 1$
zeros in $U\setminus B_\rho(0)$ for $\hat p\in \Sigma\cap B_{ \hat r}(0)$.

{Indeed, it is readily seen that 
$K$ is of class $C^1$ and $DK(u)$ is compact for all  $u$. 
Thus, $\mathcal F=I-K$ is a zero-index Fredholm operator. If $\hat p$ is a regular value, that is, $D\mathcal F(u)$ is surjective for every preimage $u \in \mathcal F^{-1}(\hat p)$ then, since the index is $0$, it is also injective and from  the open mapping theorem we conclude that $D\mathcal F(u)$ is an isomorphism. Hence,
the number
of such preimages in $U\setminus B_{\rho}(0)$
is greater or equal than
$|deg(I-K,U\setminus B_{\rho}(0),0)|$. 
This follows by taking small 
neighbourhoods $N_u$ around each of these values $u$ such that $\mathcal F:N_u\to \mathcal F(N_u)$ is a diffeomorphism. Because there are no other zeros of  $\mathcal F -\hat p$ in $U\setminus B_{\rho}(0)$, the degree is the sum of the degrees $d_u$ over each of these neighbourhoods. 
The claim then follows from the fact that  $d_u=\pm 1$ for each $u$.}

However, although the mapping 
$p\mapsto \hat p$ defined before establishes an isomorphism $J:C_T\to C_T^1$, it might happen that $J^{-1}(\Sigma\cap C^1_T)$ is not a residual set. The difficulty is overcome for example by considering the same operator $K$ as before, now defined over the set
$$
\hat U:= \{u\in C^1_T: u(t)\in \Omega,\, \|u'\|_\infty < \|g\|_{\infty} + 1\} \subset C^1_T.
$$
Details are left to the reader. 
\end{enumerate}

\begin{rem}
{Notice that}
\begin{enumerate}
    \item 
 The existence of a solution near the equilibrium can be also proved in a direct way by the Implicit Function Theorem.
 \item Condition  (\ref{hart}) alone implies the existence of generically  $|\chi(\Omega)|$
 solutions. 
 \item Analogous conclusions are obtained if the sign of  (\ref{hart}) is reversed. In this case, $G$ is homotopic to $\nu$ and hence 
 $deg(I-K,U,0)= \chi(\Omega)$. However, in this latter case the considerations about the Poincar\'e operator become less clear, because it is not guaranteed that solutions with initial values $\phi$ with $\phi(t)\in \overline \Omega$ remain inside $\Omega$.
 
\end{enumerate}
\end{rem}

{}

\section{Small delays}

\label{sec-delay}

As mentioned in the introduction, 
condition  (\ref{hart}) 
implies that the vector field $G(x)=g(x,x)$ is inwardly pointing over 
$\partial \Omega$, although the converse is not 
true; the need of a condition stronger than (\ref{hart-weak}) 
is due to the presence of 
the delay. 
However, 
if only (\ref{hart-weak}) is assumed, 
then Theorem \ref{main} 
 is still valid for all
$\tau<\tau ^*$, where $\tau ^*$ depends only on 
$\|g\|_\infty$. 
More precisely, by continuity we may 
fix $\ee>0$ such that 
(\ref{hart}) holds for all 
$x\in \partial \Omega$
and all $y\in\overline\Omega$ with 
$|y-x|<\ee$ and take 
$\tau* := \frac \ee{\|g\|_\infty}$.

In this 
section, we show that the problem for small $\tau$ can be seen as a perturbation of the non-delayed case, thus giving the explicit 
sufficient condition for the non-existence of nontrivial $T$-periodic  solutions of (\ref{linear}) expressed in Corollary \ref{smalldelay}. We shall make use of the following lemmas:

\begin{lem}
\label{lambdas}
 $1$ is a Floquet multiplier of the system $u'(t)=Mu(t)$  if and only if $-\lambda_k^2$ is 
an eigenvalue of $M^2$ for some $k\in\N_0$, that is, if and only if 
$\pm i\lambda_k$ are eigenvalues of $M$ for some $k$.  

\end{lem}

\begin{proof}
The result follows by direct computation, or from
Lemma \ref{lema} with $\tau=0$. 
\end{proof}

For example, when $M$ is triangularizable (or, equivalently, when all its eigenvalues are real), $1$ is not an eigenvalue of the system $u'(t)=Mu(t)$ 
if and only if $M$ is nonsingular; in this particular case, 
the conclusion follows directly, because the system uncouples and the result is obviously true for a scalar
equation.

\begin{lem}
\label{Floq}
Assume that 
 $1$ is not 
a Floquet multiplier
of the linear ODE system $u'(t)=(A+B)u(t)$.
Then the DDE system
(\ref{linear}) 
has no nontrivial $T$-periodic  solutions, 
provided that $\tau$ is small. 
\end{lem}

\begin{proof}
Suppose that $u_n\in C_T$ is a nontrivial solution for $\tau_n\to 0$. Without loss of generality, it may be assumed that $\|u_n\|_\infty=1$ and hence $\|u_n'\|_\infty\le C$ for some constant $C$. Thus, we may assume that $u_n$ converges uniformly to some $u\in C_T$ with $\|u\|_\infty=1$. Because $\|u_n(t-\tau_n)-u_n(t)\|\le C\tau_n\to 0$, it becomes clear that $u_n'$ converges uniformly to $(A+B)u$ which, in turn, implies $u'=(A+B)u$, a contradiction. 
\end{proof}

\begin{rem}

A more direct proof of Lemma \ref{Floq} 
follows just by considering Remark \ref{remark1}.\ref{k_0} and Lemma
\ref{lambdas}. 
Indeed, in the context of Lemma \ref{lema} 
it suffices to check that 
$h_k\ne 0$ only for a finite number of values of $k$. By continuity, this is true for small $\tau$, because ${\rm det} [(A+B)^2 + \lambda_k^2 I]\neq 0$
for all $k$. 
However, the previous proof has an interest in its own because it can be extended in a straightforward manner to the non-autonomous case.

\end{rem}

\medskip

\underline{\textit{Proof of Corollary \ref{smalldelay}}}:
As a consequence of the preceding lemma, 
the conclusions of Theorem \ref{main} hold 
for small $\tau$, provided that the linearisation 
has no nontrivial $T$-periodic solutions for the non-delayed case. Thus, in view of Lemma \ref{lambdas}, the proof
is complete.
\hfill{}$\square$

\section{Poincar\'e operator}

\label{poincare}

In this section, we shall make some considerations regarding the Poincar\'e operator associated to the system. 
Let us firstly observe that if 
$\chi(\Omega)=1$ (for example, if   $\Omega$ is homeomorphic to a ball), then the condition $s(A+B) \neq (-1)^N\chi(\Omega)$ in Theorem \ref{main} simply reads
$(-1)^N {\rm det}(A+B)<0$. This, in turn, implies 
that the 
equilibrium is unstable.
{Indeed, consider the characteristic function $h(\lambda)= {\rm det}\left(\lambda I - A - Be^{-\lambda\tau} \right)$, then $h(0)= (-1)^N {\rm det}(A+B)<0$ and $h(\lambda) =\lambda^N$ for $|\lambda|\gg 0$. In particular, this implies the existence of a characteristic value $\lambda>0$. }

We shall show that, in the present context, the 
instability of the equilibrium when
$(-1)^N {\rm det}(A+B)<0$
is due to the fact, proved in section \ref{dem},  
that the index 
of the fixed point operator $K$ at $e$
(i. e. the degree of $I-K$ 
over small balls around $e$) is equal to 
$(-1)^{N+1}$. 
When $\tau=0$, 
this can be regarded as a direct consequence of the following properties:

\begin{enumerate}

\item 
$deg(I-K, B_\rho(e),0)$ with $B_\rho(e)\subset C_T$ 
 is equal to 
 $(-1)^Ndeg_B(I-P, B_\rho(e),0)$ with $B_\rho(e)\subset\R^N$, where
 $P$ is the Poincar\'e map.

\item If the equilibrium is stable, then the index 
of $P$ 
is $1$.

\end{enumerate}

The first property is a particular case of a 
\textit{relatedness principle} due to Krasnoselskii (see \cite{krasno}). 
The second property is well-known and can be found for example in \cite{K}. For more details see \cite{rafa}, where sufficient conditions for the validity of the converse statement are also obtained.

Our goal in this section consists in understanding the connections between the instability of the equilibrium and the index of the fixed point operator defined in the proof of the main theorem.

With this aim, let us  define the Poincar\'e operator for the delayed case as follows. Let $\tau\le T$ and
consider a general autonomous system 
\begin{equation}
\label{general}    
u'(t)=F(u_t)
\end{equation}
with $F: C([-\tau,0])\to \R^N$ locally Lipschitz, \emph{i.e.}: for all  $R>0$ there exists 
a constant 
$L$ such that
$$
\|F(\phi)-F(\psi)\|\le L\|\phi-\psi\|_\infty 
$$
for all   $\phi,\psi\in \overline {B_R(0)}\subset C([-\tau,0],\mathbb R^n)$. The notation $u_t$ expresses, as usual, the mapping defined by $u_t(\theta):=u(t+\theta)$ for $\theta\in [-\tau,0]$.

Denote by ${\rm dom}(P)\subset C([-\tau,0])$ the set
of those functions $\phi$ such that the 
unique 
solution $u=u(\phi)$ of the problem with initial condition $\phi$ is defined up to $t=T$, then $P:{\rm dom}(P)\to C([-\tau,0])$ is defined by
$$
P\phi(s):=u(T+s).
$$
Clearly, the $T$-periodic solutions of 
the problem can be identified with the fixed points of $P$. 
We shall see that, as in the non-delayed case, if the linearisation has no nontrivial $T$-periodic solutions then 
the index $i(P)$ of the operator $P$ at a stable equilibrium is equal to $1$.

To this end, assume without loss of generality that $e=0$ and observe  that  stability implies that ${\rm dom}(P)$ is a neighbourhood of 
$0$. 
It is worth noticing that, 
in the general setting, extra conditions are required in order to prove the compactness of  $P$ (see \emph{e.g}. \cite{liu}), so the Leray-Schauder degree may be not well defined; 
however, it is verified  that the stability assumption implies that $P$ is compact over small neighbourhoods of $0$. More precisely:

\begin{lem}
\label{compact}

Let $F$ be as before and assume that for some open $U\subset C([-\tau,0])$ there exists  $R>0$ such that if $\phi\in U$ then the solution $u$ with initial condition $\phi$ is defined and satisfies $|u(t)| <R$ for all $t\in [0,T]$. Then  $P$ is well defined and compact over  $U$. 

\end{lem}

\begin{proof}

Let 
$B\subset U$ be bounded 
and observe, in the first place, 
that $P(B)$ is bounded. Moreover, 
if $u$ is a solution with initial condition 
$\phi\in B$, then   
$$u(t)= \phi(0) + \int_0^t F(u_s)\, ds.  
$$
Enlarging $R$ if necessary, we may assume  $B\subset B_R(0)$, then $\|u_s\|_\infty <R$ for all $s\in [0,T]$. Given $t_1<t_2$ in $[-\tau,0]$, since $\tau\le T$ it is verified that
$$|P\phi(t_2)-P\phi(t_1)| \le \int_{T+t_1}^{T+t_2} |F(u_s)|\, ds.
$$
Let $L$ be the Lipschitz constant  corresponding to $R$, then 
$$|F(\phi)|\le |F(0)| + L\|\phi\|_\infty\le C + LR,
$$
where $C:=|F(0)|$.
Hence $|P(t_2)-P(t_1)|\le (C+LR)(t_2-t_1)$ and the result follows from the  Arzel\`{a}-Ascoli Theorem. 
\end{proof}

\begin{rem}

For example, the assumptions of the previous lemma are satisfied if $F$ has linear growth, that is 
$$|F(\phi)|\le \gamma\|\phi\|_\infty + \delta.
$$
\end{rem}

Furthermore, extra assumptions are required to ensure the non-existence of
nontrivial periodic solutions near $0$; this is why we shall impose this fact as an extra
condition (see Proposition \ref{poinc-stab} below), which is clearly satisfied for example when
the stability is asymptotic.
For simplicity, 
we shall also 
assume that $F$ is Fr\'echet differentiable at $0$, that is,
$$F(\phi)= D_\phi(0)\phi + \mathcal  R(\phi) 
$$
with $\|\mathcal R(\phi)\|_\infty \le o(\|\phi\|_\infty)\|\phi\|_\infty$. Thus, it is readily verified that the linearisation of $P$ at the origin coincides with the Poincar\'e operator associated to the linearised system $u'(t)=D_\phi(0)u_t$.

\begin{prop}

\label{poinc-stab}

In the previous setting, assume that $0$ is a stable equilibrium
of (\ref{general}) such that its linearisation has no nontrivial $T$-periodic solutions. Then $i(P)=1$.

\end{prop}

\begin{proof}
Without loss of generality, we may assume that $P$ is compact on $\overline V$ for some neighbourhood $V$ of $0$. It follows from the assumptions that the index of $P$ is well defined and coincides with the index of  its linearisation $P_L$. 
According to Theorem 13.8 in  
\cite{brown}, 
$deg(I-P_L,B_\rho(0),0)$ is equal to $(-1)^\alpha$, where 
$\alpha$ 
is the sum of the (finite) algebraic  multiplicities of the (finitely many) eigenvalues $\sigma$ of $P_L$ satisfying $\sigma>1$. 

If $deg(I-P_L,B_\rho(0),0) = -1$, 
then  $P_L$ has 
an eigenfunction $\phi$ with eigenvalue $\sigma>1$.
If $u$ is the 
 corresponding solution of the linearised problem with initial 
 condition $u=\phi$ on $[-\tau,0]$
then $u$ can be extended to $\R$ in a $(T,\sigma)$-periodic 
fashion, that is, with $u(t+T)=\sigma u(T)$ for all $t$ (see \cite{pinto}). 
In particular, 
$u(t)$ is unbounded 
for $t>0$. In other words, $0$ is unstable for the linearised problem which, in turn, implies that it cannot be  stable for the original problem (see \emph{e.g.} \cite{hale}).

\end{proof}

In order to complete the picture for system (\ref{ec}),
it would be interesting to  prove that, indeed, the index of the Poincar\'e operator at the equilibrium when the linearisation has no nontrivial solutions is $(-1)^Ns(A+B)= (-1)^Ni(K)$. Here, we shall simply verify that the claim holds when the delay is small; the analysis of the general case and  
 {a version of the Krasnoselskii relatedness principle for delayed systems shall be the 
subject of a forthcoming paper.}

To this end, let us start with a
direct computation for the non-delayed case:

\begin{lem}
\label{degP}

Let $M\in \mathbb R^{N\times N}$ and let $P_M$ be the Poincar\'e operator associated to the linear ODE system $u'(t)=Mu(t)$
for some fixed $T$. If   $1$ is not 
a Floquet multiplier, then
$$deg_B(I-P_M,V,0) = (-1)^Ns(M)$$
for any neighbourhood $V\subset \R^N$ of the origin.

\end{lem}

\begin{proof}

By definition,
$$(I-P_M)(u)= \left(I-e^{TM}\right)u.$$  
Write $M$ in its (possibly complex) Jordan form $M=C ^{-1}JC$, where $J$ is upper triangular. Then 
$${\rm det}\left(I- e^{TM}\right) = 
{\rm det}\left(I- e^{TJ}\right) = \prod_{j=1}^N 
\left(1-e ^{\lambda_jT}\right),
$$
where $\lambda_j$ are the eigenvalues of $M$. Now observe that if $\lambda=a+ib\notin \R$, then 
$$
\left(1-e ^{\lambda T}\right)\left(1-e ^{\overline{\lambda}T}\right) = 1 + e ^{aT}\left(e^{aT} -2\cos (bT)\right) >0. 
$$
Thus, complex eigenvalues do not 
affect 
the sign of ${\rm det}\left(I- e^{TM}\right)$, as well as it happens with the sign of ${\rm det}(M)$ because $\lambda\overline\lambda =|\lambda|^2$. The result follows now from the fact that, for $\lambda\in\R$,  
$$sgn\left(1 - e^{\lambda T}\right) = -sgn (\lambda).$$

\end{proof}

\begin{rem}
An  alternative (somewhat exotic) proof follows from the relatedness principle. 
Indeed, we may consider the operator $K_L$ in the proof of Theorem 
\ref{main} with $A=M$ and $B=0$, then $deg_B(I-P,V,0)=(-1)^Ndeg(I-K_L, V,0) = (-1)^Ns(M)$. 

\end{rem}

The conclusion for small $\tau$ is obtained now by a continuity argument. 
Indeed, fix $r>0$ and $P_L$
as before. The solutions of (\ref{linear}) with initial value $\phi\in B_r(0)$ are uniformly bounded; thus, by Gronwall's lemma we deduce that $\|P-P_0\| = O(\tau)$, where the operator $P_0$ is defined by $P_0(\phi)(t)\equiv v(T)$, with $v$ the unique solution of the system $v'(t)=(A+B)v(t)$ satisfying $v(0)=\phi(0)$. 
Moreover, recall that if $\tau$ is small then 
$P_L$ is homotopic 
to $P_0$; thus, the result follows from Lemma \ref{degP}.

\section{Example: a system of DDEs with singularities}
\label{exam}

A simple example is presented here in order to illustrate our main results. 
Let $0\le J_0\le J \ne 0$ and
$$
g(x,y):= -dx + |y|^2\left( 
\sum_{j=1}^{J_0} a_j\frac{x-v_j}{|x-v_j|^{\alpha_j}} +
\sum_{j=J_0+1}^{J} a_j\frac{y-v_j}{|y-v_j|^{\alpha_j}}
\right)$$
where $d,a_j>0$, $\alpha_j>2$
and $v_j\in \R^N\backslash\{0\}$ are pairwise different vectors. 
A simple computation shows that
$$\langle g(x,x),x\rangle < 0 
\qquad |x|\gg 0$$ 
and
$$\langle g(x,x),v_j-x\rangle < 0 \qquad |x-v_j|\ll 1 
$$
for $j=1,\ldots, J$. Moreover, $g(0,0)=0$ and
$$A=D_xg(0,0)=-dI, \quad B=D_yg(0,0)=0. 
$$
Thus, taking $\Omega:=B_R(0)\backslash \cup_{j=1}^J B_\eta(v_j)$ where 
$R\gg 0$ and $\eta\ll 1$, 
Corollary \ref{smalldelay} applies. Since 
$\chi(\Omega)= 1-J < 1 = (-1)^Ns(A+B)$,
we conclude that the number of 
$T$-periodic solutions of (\ref{nonaut})
for small $\tau$ and $\|p\|_\infty$ is generically $J+1$.

\section*{Acknowledgements}

The first two authors were partially supported by projects CONICET PIP 11220130100006CO 
and UBACyT 20020160100002BA.

The first author wants to thank Prof. J. Barmak for his thoughtful comments regarding the fixed point property  and the Euler characteristic.

\bigskip

\noindent Pablo Amster and Mariel Paula Kuna 

\medskip
\noindent \textit{E-mails}: pamster@dm.uba.ar --
mpkuna@dm.uba.ar.

\medskip

\noindent Departamento de Matem\'atica, Facultad de Ciencias Exactas y Naturales, Universidad de Buenos Aires, Ciudad Universitaria, Pabell\'on I, 1428 Buenos Aires, Argentina and IMAS-CONICET.

\bigskip
\noindent Gonzalo Robledo 

\medskip
\noindent \textit{E-mail}: grobledo@uchile.cl.

\bigskip

\noindent Departamento de Matem\'aticas, Facultad de Ciencias,  Universidad de Chile,  Casilla 653 Santiago, Chile.

\end{document}